\documentclass[11pt, reqno,fleqn]{article}
\usepackage[utf8]{inputenc}
\usepackage{amsmath}
\usepackage{amssymb}
\usepackage{mathrsfs}
\usepackage{amsthm}
\usepackage{mathtools}
\usepackage{cite}
\usepackage{hyperref}
\usepackage{enumerate}
\usepackage[left=1in,right=1in,top=1.3in,bottom=1.3in]{geometry}
\usepackage{titlesec}
\titleformat*{\section}{\small\bfseries}
\newtheorem{theorem}{Theorem}[section]
\newtheorem{lemma}[theorem]{Lemma}
\newtheorem{corollary}[theorem]{Corollary}

\newtheorem{proposition}[theorem]{Proposition}

\theoremstyle{definition}
\newtheorem{remark}[theorem]{Remark}
\newtheorem{example}[theorem]{Example}

\numberwithin{equation}{section}
\author{\large Anuradha Gupta and Kajal Negi}
\date{}

\begin{document}
\title{Hyponormality of the sum of two Toeplitz operators }
\maketitle

\begin{abstract}
In this paper, we have studied the hyponormality and invertibility of the operator of type $wT_{\varphi}+T_{\psi}$ where $w$ is any non-zero complex number and $T_{\varphi}, T_{\psi} $ are Toeplitz operators. We have also studied hyponormality when the symbol of  Toeplitz operator is a linear functional.
\end{abstract}



\section{Introduction}
Let $dA(z)$ denote the normalized area measure on the open unit disk $\mathbb{D}$.
Consider the Hilbert space \( L^2(\mathbb{D}) \) consisting of square-integrable measurable functions on \( \mathbb{D} \) equipped with the inner product
$$
\langle f, g \rangle = \frac{1}{2\pi} \int_{\mathbb{D}} f(z)\, \overline{g(z)} \, dA(z).
$$
The Hardy space \( H^2(\mathbb{D}) \) consists of all analytic functions on \( \mathbb{D} \) whose power series coefficients are square-summable. Let \( L^\infty(\mathbb{D}) \) denote the space of all bounded measurable functions on \( \mathbb{D} \).
For $\varphi \in L^{\infty}(\mathbb{D})$, the Toeplitz operator $T_{\varphi}$ is defined  as $$T_{\varphi}(f)=P(\varphi f) \  \text{ where $f \in H^2(\mathbb{D})$}$$
and $P$ is the orthogonal projection from $L^2(\mathbb{D})$ onto $H^2(\mathbb{D})$. 
The Hardy space $H^2$ is a  reproducing kernel Hilbert space. The reproducing kernel of Hardy space is denoted by $K_w$ for $w \in \mathbb{D}$ and is defined as $$K_w(z)= \frac{1}{1-\overline{w}z}  \text{ for all $ z \in \mathbb{D}$}.$$  
Its norm is given by $$\|K_w\|= \sqrt{K_{w}(w)}= \sqrt{\frac{1}{1-|{w}|^2}} \text{ for all $w \in \mathbb{D}$}.$$
For any holomorphic self-map \( f \) of the open unit disk \( \mathbb{D} \), there exists a unique point \( \xi \) in the closed disk \( \overline{\mathbb{D}} \) such that the iteration \( f^n(z) \) converges to \( \xi \) for any point \( z \) in \( \mathbb{D} \) . This point \( \xi \), known as the Denjoy-Wolff point, is either a fixed point of \( f \) within \( \mathbb{D} \) or a boundary point of \( \mathbb{D} \) if no fixed point exists inside \( \mathbb{D} \).
Consider \( H \) as a separable complex Hilbert space  and let \( L(H) \) represents the algebra consisting of all bounded linear operators on \( H \).
An operator $T \in L(H) $ is considered to be  normal if $T$ and $T^*$ commutes and hyponormal if the self-commutator   $[T^*, T]:=T^*T-TT^*$ is positive. 
Kim and Ko \cite{article4}  studied the hyponormality of the sum of composition operators on the Hardy space.  Similarly, Gupta and Aggarwal \cite{MR4728014} explored the hyponormality of the sum of Toeplitz operators with non-harmonic symbols on the Fock space and it was shown that the sum of Toeplitz hyponormal operators is not always hyponormal. Motivated by their results, we have investigated the hyponormality of the sum of composition-differentiation and Toeplitz operators on the Hardy space. 
These operators are utilized in signal processing for tasks like noise reduction and image processing and also play pivotal roles in prediction theory, aiding in forecasting based on historical data, as well as in wavelet analysis for feature extraction and differential equations for solving boundary value problems. 
We begin with the following known property of Toeplitz operators which is instrumental in the subsequent results:
\begin{proposition} \cite{MR2687747}
 If $f, g \in L^{\infty}(\mathbb{D})$,  then the following equivalences hold :\\
 \text{1.}  $T_{f+g} = T_{f}+ T_{g}. $\\
2. $T_{f}^* =T_{\overline{f}}.$\\
3. $T_f T_g =T_{fg}$ if f or g is analytic.
\end{proposition}
\begin{lemma} \label{Lemma 1.2.}
Let $t$ and $s$ be non-negative integers and  $ w \in \mathbb{D} $,  then 
\begin{equation}
P(\overline{w} ^t w^s)=
             \begin{cases}
\frac{1}{2(s+1)} w^{s-t}  & \text{ for s $\geq$ t} \notag             , \\
0                               & \text{for s $<$ t}  \notag
\end{cases}
\end{equation}
where $P$ is orthogonal projection from $L^{2}(\mathbb{D})$ onto $H^{2}{(\mathbb{D})}$.
\end{lemma}
\begin{proof}
For non-negative integers $s$ and $t$ and $ z,w \in \mathbb{D}$,  if $s \geq  t$ then  
 \begin{align*}
P(\overline z ^{t} z^s) = \langle P(\overline{w}^t w^s), K_{z}( w) \rangle \notag
 =  \frac{1}{2\pi}\int (\overline{w}^t w^s )\sum\limits_{n=0}^{\infty}  (z \overline{w} )^n  \ dA(w).
\end{align*}
 As the series $\sum\limits_{n=0}^{\infty} (z \overline{w} )^n $ converges uniformly, therefore,  summation and integration can be interchanged. 
$$P(\overline z^{t} z^s) =\frac{1}{2\pi}\sum\limits_{n=0}^{\infty} z^n \int \overline{w}^t w^s \overline{w}^n   dA(w).$$
                   Let $w=re^{i\theta}$ for $0 \leq \theta \leq 2\pi$ and $0<r <1, $ 
\begin{align} \label{(1.2)}
P(\overline{z}^{t} z^s) &= \frac{1}{2\pi} \sum\limits_{n=0}^{\infty} {z^n} \int\limits_{r=0}^{1} \int\limits_{\theta =0}^{2 \pi} r^{t+s+n+1} e^{i(s-t-n)\theta} \ d{\theta}  \ d{r} \\ 
 &\nonumber= \frac{1}{2(s+1)} z^{s-t}
.\end{align}
 For $  s< t $ and  $w \in \mathbb{C},   \ 
P(\overline{z}^{t} z^s) = 0 $  as $n$ is a non-negative integer. Therefore, in equation \ref{(1.2)} ,
$ \int\limits_{\theta =0}^{2 \pi}  e^{i(s-t-n)\theta}  \ d{\theta} =0. \text{ Hence,   }  P(\overline{z}^{t} z^s) = 0.$
\end{proof}
\begin{remark}
\text{ For strictly positive integers }   $s, t  \text{ and } z \in \mathbb{D}.$ \text{ Consider}
\begin{align*}
    \langle z^s , z^t \rangle &= \frac{1}{2\pi} \int_{\mathbb{D}} z^s \overline{z}^t dA(z) = \frac{1}{2\pi} \int_{0}^{1} \int_{0}^{2\pi} r^{s+t+1} e^{i(s-t)\theta} dr \  d\theta .
    \end{align*}
    \begin{align*}
\text{Therefore,  } \langle z^s,  z^t \rangle = 
    \begin{cases}
    \frac{1}{2(s+1)} ,  & \text{if } s=t,  \\
    0    ,   & \text{otherwise.}                        \label{(1.3)}
\end{cases}   
\end{align*}
\end{remark}

\begin{theorem}\label{theorem 1.2}\cite{article5}
    If \( \phi \in L^\infty(\partial \mathbb{D}) \), where \( \phi = f + \overline{g }\) for \( f, g \in H^2 (\mathbb{D} )\), then the operator \( T_\phi \) is hyponormal if and only if there exists a constant \( c \) and a function \( h \in H^\infty( \mathbb{D}) \) such that \( g = c + T_h f \), where \( \|h\|_\infty \leq 1 \).
\end{theorem}
\begin{theorem}\label{theorem 1.3}\cite{Allen}
Let \( \phi \in L^\infty(\partial \mathbb{D}) \) then the Toeplitz operator \( T_\phi \) on the Hardy space \( H^2(\mathbb{D} ) \) is invertible if and only if \( \phi \) is invertible in \( L^\infty(\mathbb{D}) \) and the unimodular function \( \frac{\phi}{|\phi|} \) admits a representation
\[ \frac{\phi}{|\phi|} = e^{(\xi + \overline{\eta} + c)}, \]
where \( \xi \) and \( \eta \) are real functions in \( L^\infty(\mathbb{D}) \), \( c \in \mathbb{R} \).
\end{theorem}
\section{Hyponormality}
In this section, we will study some results on the hyponormality  of the sum of two Toeplitz operators of the type $wT_{\varphi}+T_{\psi}$ where $w$ is non-zero complex number and the functions  $\varphi$, $\psi$ are in $L^{\infty}(\mathbb{D} )$.
\begin{theorem}\label{Theorem 2.1.}
    Let $\varphi$, $ \psi$ $\in L^{\infty}(\mathbb{D})$ be such that $w T_{\varphi}+ T_{\psi}$ is hyponormal for any non-zero $w \in \mathbb{C}$ then $T_{\varphi}$ and $T_{\psi}$ are hyponormal and 
$$   | \langle [T^{*}_{\varphi}, T_{\psi}] h,h \rangle |^2 \leq \langle [T^{*}_{\varphi}, T_{\varphi}] h,h \rangle  \langle [T^{*}_{\psi}, T_{\psi}] h,h \rangle   \hspace{1cm}   \text{ for all  h } \in H^2(\mathbb{D} ).$$
Conversely, if $T_{\varphi}$ and $T_{\psi}$ is  hyponormal and $ Re\{\overline{w} \langle [T^*_{\varphi},T_{\psi} ]h, h \rangle \} \geq 0  $ for all h $\in H^2(\mathbb{D} )$ and $ 0 \neq w \in \mathbb{C}$, then  $wT_{\varphi} + T_{\psi}$ is hyponormal.
\end{theorem}
\begin{proof}
    Let $\varphi$, $ \psi$ $\in L^{\infty}(\mathbb{D})$ and $wT_{\varphi} + T_{\psi}$ is hyponormal for any non-zero $w \in \mathbb{C}$ 
    \begin{align}
        &\nonumber\implies [(wT_{\varphi} + T_{\psi})^*, wT_{\varphi} + T_{\psi}] \geq 0\\
         & \nonumber \iff (wT_{\varphi} + T_{\psi})^* (wT_{\varphi} + T_{\psi})  \geq (wT_{\varphi} + T_{\psi})(wT_{\varphi} + T_{\psi})^*\\
         &\nonumber \iff (\overline{w}T^*_{\varphi} + T^*_{\psi})(wT_{\varphi} + T_{\psi}) \geq (wT_{\varphi} + T_{\psi})(\overline{w}T^*_{\varphi} + T^*_{\psi})\\
         &\nonumber \iff |w|^2(T^*_{\varphi}T_{\varphi}-T_{\varphi}T^*_{\varphi})+\overline{w} (T^*_{\varphi}T_{\psi}-T_{\psi}T^*_{\varphi}) + w(T^*_{\psi}T_{\varphi}-T_{\varphi}T^*_{\psi})+(T^*_{\psi}T_{\psi}-T_{\psi}T^*_{\psi})\geq 0\\
        & \label{2.1} \iff |w|^2 [T^*_{\varphi},T_{\varphi}]+\overline{w}[T^*_{\varphi},T_{\psi}]+w[T^*_{\psi},T_{\varphi}]+[T^*_{\psi},T_{\psi}] \geq 0.
    \end{align}
    Since $w \neq 0$, let $w =r e^{i \theta}$, $r>0$ and $\theta \in [0,2\pi]$, 
    \begin{align}
        & \iff r^2 [T^*_{\varphi},T_{\varphi}]+re^{-i \theta} [T^*_{\varphi}, T_{\psi}] + re^{i \theta} [T^*_{\psi}, T_{\varphi}] +[T^*_{\psi}, T_{\psi}] \geq 0 \nonumber\\
        & \iff [T^*_{\varphi},T_{\varphi}] + \frac{1}{r^2}[T^*_{\psi},T_{\psi}]+ \frac{e^{-i \theta}}{r} [T^*_{\varphi},T_{\psi}]+\frac{e^{i \theta}}{r}[T^*_{\psi},T_{\varphi}] \geq 0
    \end{align} 
    Letting $r \rightarrow  \infty $ we have
    $[T^*_{\varphi}, T_{\varphi}] \geq 0 $ $\implies T_{\varphi}$ is hyponormal.
    Since $[T^*_{\psi} , T_{\varphi}]=[T^*_{\varphi} ,T_{\psi}]^*.$
    \begin{align*}
      &  |w|^2[T^*_{\varphi} ,T_{\varphi}]+[T^*_{\psi} ,T_{\psi}]+\overline{w}[T^*_{\varphi} ,T_{\psi}]+w[T^*_{\psi} ,T_{\varphi}] \geq 0.\\
      & \implies |w|^2[T^*_{\varphi} ,T_{\varphi}]+[T^*_{\psi} ,T_{\psi}] +2 Re \{\overline{w}[T^*_{\varphi} ,T_{\psi}]\} \geq 0.\\
      & \implies |w|^2 \langle [T^*_{\varphi} ,T_{\varphi}]h,h\rangle +\langle [T^*_{\psi} ,T_{\psi}] h, h \rangle +  2 Re \{\overline{w} \langle [T^*_{\varphi} ,T_{\psi}]h,h\rangle \} \geq 0  \ \ \text{ for all } h \in H^2(\mathbb{D} ). 
    \end{align*}
    Since  $[T^*_{\varphi} ,T_{\varphi}]$ and $[T^*_{\psi} ,T_{\psi}]$ are self-adjoint both $\langle [T^*_{\varphi} ,T_{\varphi}] h, ,h \rangle $ and $\langle [T^*_{\psi} ,T_{\psi}] h, h \rangle $ are real for all $h \in H^2(\mathbb{D} )$. Therefore, from \cite{article3}\\
$$|\langle [T^*_{\varphi},T_{\psi}]h,h\rangle |^2 \leq \langle [T^*_{\varphi},T_{\varphi}]h,h\rangle \langle [T^*_{\psi},T_{\psi}]h,h\rangle \text{ for all } h \in H^2(\mathbb{D} )$$
    Since $[ T^*_{\varphi},T_{\varphi} ] $ is positive. Hence, $\langle [T^*_{\psi},T_{\psi}]h,h  \rangle \geq 0 $ for all $h \in {H}^2(\mathbb{D} ).$ \\
    Conversely, assume the operators $T_{\varphi}$ and $T_{\psi}$ are hyponormal then $ [T^*_{\varphi},T_{\varphi}],  [T^*_{\psi},T_{\psi}] \geq 0.$\\
    Consider
    \begin{align}
       \nonumber& [(w T_{\varphi}+T_{\psi})^* , w T_{\varphi}+T_{\psi}]\\
       \nonumber&\hspace{2cm}= (w T_{\varphi}+T_{\psi})^*(w T_{\varphi}+T_{\psi})- (w T_{\varphi}+T_{\psi})(w T_{\varphi}+T_{\psi})^*\\
       &\nonumber \hspace{2cm}=|w|^2 [T^*_{\varphi},T_{\varphi}]+ [T^*_{\psi},T_{\psi}] + 2 Re\{ \overline{w} [T^*_{\varphi},T_{\psi}]\}\\
       &\hspace{2cm}=|w|^2  \langle [T^*_{\varphi},T_{\varphi}] h, h \rangle + \langle [T^*_{\psi},T_{\psi}] h, h \rangle  + 2 Re\{ \overline{w} \langle [T^*_{\varphi},T_{\psi}] h, h \rangle \} \text{ for all $h \in H^2(\mathbb{D} ).$} \label{(2.3)}
    \end{align}
If $Re \{\overline{w} \langle [T^*_{\varphi},T_{\psi}] h, h \rangle \} >0 $ for all $h \in H^2(\mathbb{D} )$ 
then $ [(w T_{\varphi}+T_{\psi})^*, w T_{\varphi} +T_{\psi} ] \geq 0. $
Hence, the operator $w T_{\varphi}+T_{\psi} $ is hyponormal.
     \end{proof}
     For $\varphi$, $ \psi$ $\in L^{\infty}(\mathbb{D})$. 
     If we drop the condition of  $Re \{ \bar{w} \langle [T^*_{\varphi}, T_{\psi}] h, h \rangle \}  >0 $ for all $h \in H^2(\mathbb{D} )$ in Theorem \ref{Theorem 2.1.}  then the operator $w T_{\varphi}+T_{\psi}$ for non-zero $w \in \mathbb{C}$ is not hyponormal. The following examples verifies the same:
     \begin{example}
         Let $\varphi(z) = \frac{z^2}{4}+\frac{3}{4} $, $\psi(z) = \frac{z^2}{3}+\frac{z}{3}+\frac{ 1}{3}$, $ w = \frac{-1}{2}$ and $f(z)= z$. Consider
         \begin{align}
         \langle  T_{\psi}f(z) ,  T_{\varphi}f(z) \rangle &= \bigg \langle \frac{z^3}{3}+\frac{z^2}{3}+\frac{z}{3} , \frac{z^3}{4} + \frac{3z}{4} \bigg \rangle
          = \frac{7}{96}.\label{.4}\\
           \langle  T^*_{\varphi}f(z) ,  T^*_{\psi}f(z) \rangle &= \bigg \langle   \frac{3 z}{4}  ,  \frac{1}{12} + \frac{z}{3} \bigg \rangle = \frac{1}{16}.\label{.5}
         \end{align}
         From  \ref{.4} and \ref{.5} we have $Re \{ \bar{w} \langle [T^*_{\varphi}, T_{\psi}] f(z), f(z) \rangle \}  = \frac{-1}{192} < 0 $. Consider,
         \begin{align*}
           &  \langle T_{\varphi} f(z),T_{\varphi}f(z) \rangle - \langle T^*_{\varphi} f(z),T^*_{\varphi} f(z)\rangle  = \bigg \langle \frac{z^3}{4} + \frac{3z}{4} , \frac{z^3}{4}+ \frac{3z}{4} \bigg \rangle - \bigg\langle \frac{3z}{4}  , \frac{3z}{4} \bigg \rangle
             = \frac{1}{128}.\\
           &  \langle T_{\psi}f(z) ,T_{\psi} f(z)\rangle - \langle T^*_{\psi}f(z) ,T^*_{\psi}f(z) \rangle = \bigg \langle \frac{z}{3}+\frac{z^2}{3}+\frac{z^3}{3}, \frac{z}{3}+\frac{z^2}{3}+\frac{z^3}{3} \bigg \rangle  - \bigg \langle \frac{1}{12} + \frac{z}{3} , \frac{1}{12}+\frac{z}{3}\bigg \rangle = \frac{25}{864}.
         \end{align*}
         From equation \ref{(2.3)}, we have, $[(w T_{\varphi}+T_{\psi})^* , w T_{\varphi}+T_{\psi}] =-0.03162< 0 .$ Thus, the operators $T_{\varphi} $ , $T_{\psi} $ are hyponormal but  $Re \{ \bar{w} \langle [T^*_{\varphi}, T_{\psi}] f(z), f(z) \rangle \} < 0  $. Therefore,
 the operator $wT_{\varphi}+T_{\psi} $ is not hyponormal. 
     \end{example}
     \begin{example}
         Let $\varphi \in L^{\infty}(\mathbb{D})$ be such that the operator $T_{\varphi}$ is hyponormal. For $w=-2$ the $Re \{ \bar{w} \langle [T_{\varphi}, T_{\varphi}] h, h \rangle \}  $ is negative for all $h \in H^2(\mathbb{D} ).$ On calculating $[(w T_{\overline{\varphi}}+T_{\varphi})^* , w  T_{\overline{\varphi}}+T_{\varphi}]$ we get,
         \begin{align*}
             &(T_{\varphi}-2T_{\overline{\varphi}})^*(T_{\varphi}-2T_{\overline{\varphi}}) - (T_{\varphi}-2T_{\overline{\varphi}})(T_{\varphi}-2T_{\overline{\varphi}})^* =  -3 (T_{\overline{\varphi}}T_{\varphi} - T_{\varphi} T_{\overline{\varphi}} ) < 0. 
         \end{align*}
         Hence, the operator $-2T_{\overline{\varphi}}+T_{\varphi}$ is not hyponormal.
     \end{example}
     \begin{theorem}
         Let $\varphi \in L^{\infty}(\mathbb{D})$. The operator $T_{(w+1){\varphi}}$ is hyponormal for any non-zero $w \in \mathbb{C} $ if and only if $T_{\varphi}$ is hyponormal. 
     \end{theorem}
     \begin{proof}
         For $\varphi \in L^{\infty}(\mathbb{D})$. Replace $ \psi$ by $\varphi$ in Theorem \ref{Theorem 2.1.}. If the  operator $wT_{\varphi}+ T_{\varphi}$ is hyponormal for any non-zero $w \in \mathbb{C}  $ then from Theorem \ref{Theorem 2.1.} the operator $T_{\varphi}$ is hyponormal.\\
         Conversely, let $T_{\varphi}$ be hyponormal where $\varphi \in L^{\infty}(\mathbb{D})$ and the operator $wT_{\varphi}+ T_{\varphi}$ is not  hyponormal for any non-zero $w \in \mathbb{C}$, then from equation \ref{(2.3)}, we have,
         \begin{align}
             [(wT_{\varphi}+T_{\varphi})^*, wT_{\varphi}+T_{\varphi} ] &= (1+|w|^2) [T_{\varphi}^*, T_{\varphi}] + 2 Re \{ \overline{w} [T_{\varphi}^*, T_{\varphi}] \} < 0. \label{(2.4)}
         \end{align}
         $\text{Hyponormality of the operator } T_{\varphi} \text{ gives } [T_{\varphi}^* , T_{\varphi}]  \geq 0.
            \text{ From equation \ref{(2.4)} , }  1+|w|^2+ 2 Re \{ w\} < 0.$
         Let $w = x+ iy $ where $x, y \in \mathbb{R}  $.
         \begin{align*}
             &\implies 1+x^2 +y^2+ 2x  < 0\\
             & \implies (1+x)^2 + y^2 < 0,
         \end{align*}
which is a contradiction as sum of two non-negative numbers is non-negative. Therefore, $wT_{\varphi} + T_{\varphi}$ is hyponormal for $w \in \mathbb{C}$.
     \end{proof}
     \begin{theorem}
         Let $\varphi \in L^{\infty}(\mathbb{D})$. The operator $wT_{\varphi}+T^*_{\varphi}$ is hyponormal for any non-zero $w \in \mathbb{C} $ if and only if $T_{\varphi}$ is normal. 
     \end{theorem}
     \begin{proof}
           For $\varphi \in L^{\infty}(\mathbb{D})$. Replace $ T_{\psi}$ by $T^*_{\varphi}$ in Theorem \ref{Theorem 2.1.}. If the  operator $wT_{\varphi}+ T^*_{\varphi}$ is hyponormal for any non-zero $w \in \mathbb{C} $ 
        then the operators $T_{\varphi}$ and $T^*_{\varphi}$ are hyponormal. Therefore, the operator $T_{\varphi}$ is normal.\\
        Conversely, if the operator $T_{\varphi}$ is normal then from the equation \ref{(2.3)} the operator $w T_{\varphi}+T^*_{\varphi}$ is normal. As normality impiles hyponormality. Therefore, the operator $w T_{\varphi}+T^*_{\varphi}$ is hyponormal.
     \end{proof}
     \begin{corollary}
         Let $\varphi \in L^{\infty}(\partial  \mathbb{D})$ such that $\varphi = f+ \overline{g}$ , $f, g \in H^2(\mathbb{D} )$. If $wT_{\varphi}+T^*_{\varphi}$ is hyponormal for $w \neq 0 \in \mathbb{C}$ then $\varphi(z)=c+f+T_{\overline{h}}f$ for some constant $c$ and some function $h \in H^{\infty}(\mathbb{C})$ with $\| h\|_{\infty} \leq 1.$
         \end{corollary}
         \begin{proof}
         If we replace $T_{\psi}$ by $T^*_{\varphi}$ in Theorem \ref{Theorem 2.1.} then $T_{\varphi}$ and $T^*_{\varphi}$ are hyponormal. From Theorem \ref{theorem 1.2}, $\varphi(z)=c+f+T_{\overline{h}}f$ for some constant $c$ and some function $h \in H^{\infty}(\mathbb{C})$ with $\| h\|_{\infty} \leq 1.$
     \end{proof}
     \begin{theorem}
         Let $\varphi$ and $\psi$  be analytic maps from $\mathbb{D}$ into itself. If the operator $wT_{\varphi}+T_{\psi}$ is hyponormal for any non-zero $w\in \mathbb{C}$ then 
$$|\langle[T^*_{\varphi} , T_{\psi}]K_{\alpha} , K_{\alpha} \rangle |^2 \leq \|\varphi\|^2 \|\psi\|^2 \bigg(\frac{1}{1-|\alpha|^2}\bigg)^2 \  \text{for any } \alpha \in \mathbb{D}.$$
     \end{theorem}
     \begin{proof}
         Suppose $wT_{\varphi}+T_{\psi}$ is hyponormal for $0 \neq w \in \mathbb{C}.$
         Consider for any $\alpha \in \mathbb{D}$
         \begin{align*}
         \langle [T^*_{\varphi},T_{\varphi}]K_{\alpha}, K_{\alpha} \rangle &= 
         \langle (T^*_{\varphi}T_{\varphi} - T_{\varphi} T^*_{\varphi})K_{\alpha}, K_{\alpha} \rangle\\
         &=\|T_{\varphi}K_{\alpha}\|^2 - \|T^*_{\varphi}K_{\alpha}\|^2\\
        &\leq \|T_{\varphi}\|^2\|K_{\alpha}\|^2\\
     &= \frac{\|\varphi\|^2}{1-|\alpha|^2}.
     \end{align*}
     Similarly, $\langle[T^*_{\psi},T_{\psi}K_{\alpha}, K_{\alpha}\rangle \leq \frac{\|\psi\|^2}{1-|\alpha|^2}$  for any $\alpha \in \mathbb{D}$.\\
     Thus, from Theorem \ref{Theorem 2.1.}
     $|\langle[T^*_{\varphi} , T_{\psi}]K_{\alpha} , K_{\alpha} \rangle |^2 \leq \|\varphi\|^2 \|\psi\|^2 \bigg(\frac{1}{1-|\alpha|^2}\bigg)^2 $ for any $\alpha \in \mathbb{D}$.
     \end{proof}
A closed subspace \( S \) of a Hilbert space \( H \) is called an \textit{invariant subspace} for an operator \( T \in L(H) \) if it satisfies \( T(S) \subseteq S \).  For a positive integer \( m \) and a point \( \alpha \in \mathbb{D} \), the \textit{\( m \)-th derivative evaluation kernel} at \( \alpha \), denoted by \( K^{[m]}_{\alpha} \), is defined as the function in the Hardy space \( H^2(\mathbb{D}) \) such that  
\(
\langle f, K^{[m]}_{\alpha} \rangle = f^{(m)}(\alpha)
\)  
for every function \( f \in H^2(\mathbb{D}) \).

     \begin{theorem}
         Let $\varphi$, $\psi$ be analytic self map of $\mathbb{D}$ and $c$ be the Denjoy-Wolff point of $\varphi$ and $\psi $ in $\mathbb{D} $. If $m$ is  a positive integer then, 
         $S_{m}(c) := span\{K_{c},K^{[1]}_{c},...,K^{[m]}_{c}\} $ is an invariant subspace of $(w T_{\varphi} + T_{\psi})^*$ for any non-zero $w \in \mathbb{C}$.
     \end{theorem}
     \begin{proof}
     Let $c$ be the Denjoy-Wolff point of $\varphi$ and $\psi $ in $\mathbb{D} $. Consider
         \begin{align*}     (wT_{\varphi}+T_{\psi})K_{c}&=wT_{\varphi}K_{c} +T_{\psi}K_{c}\\ 
         &=(w\varphi +\psi)K_{c}.
\end{align*}
For positive integer $n$ and  any $f \in H^2(\mathbb{D} )$. Consider
\begin{align*}
\langle f,(w T_{\varphi}+T_{\psi})^*K_{c}^{[n]}\rangle 
&= \langle(w T_{\varphi}+T_{\psi})  f,K_{c}^{[n]}\rangle \\
&= \frac{d^n}{dz^n} [w\varphi(z) f(z) +\psi(z)f(z)]\bigg|_{z=c}\\
&=\langle f, \sum_{i=0}^{n} \overline{w\varphi^{i}(c)+\psi^{i}(c)}K_{c}^{[n-i]}\rangle .
\end{align*}
Therefore, $(wT_{\varphi}+T_{\psi} )^* K_{c}^{[n]}=\sum_{i=0}^n (\overline{w} \overline{\varphi^i}(c)+\overline{\psi^i}(c))K_{c}^{[n-i]}.$
Hence, $ S_{m}(c)$ is an invariant subspace of $(w T_{\varphi}+T_{\psi})^*$ and  $S_{m}(c)^ {\perp}$ is an invariant subspace of $(w T_{\varphi}+T_{\psi})$.
     \end{proof}
     \begin{theorem}
    Let \( \varphi, \psi \in L^\infty(\partial \mathbb{D}) \) and the set $\{ \varphi,  \psi \}$ is linearly independent. The operator $w T_{\varphi}+ T_{\psi} $ is invertible for any non-zero $w\in \mathbb{C}$  and $\varphi$ , $\psi$ is invertible if and only if  $T_{\varphi}$, $T_{\psi}$  are invertible operators.  Also, $\frac{(\varphi+w\psi)}{|\varphi+w\psi|}  \leq e^{\xi_{1}+\overline{\eta_{1}}+c_{1}}+\frac{w}{|w|} e^{\xi_{2}+\overline{\eta_{2}}+c_{2}}$ where  $c_{1}, c_{2} \in \mathbb{R}$ and $\xi_{1},\xi_{2},\eta_{1},\eta_{2}\in L^{\infty}(\mathbb{D})$
     \end{theorem}
     \begin{proof}
     Let \( \varphi, \psi \in L^\infty(\partial\mathbb{D}) \) from Theorem \ref{theorem 1.3} and for non-zero $w \in \mathbb{C}$, the operator $w T_{\varphi}+ T_{\psi} $ is invertible  
       if and only if $(w {\varphi} +{\psi} )$ is invertible. If   $\varphi$ and $\psi $  are invertible then from Theorem \ref{theorem 1.3} the operators $T_{\varphi}$ and $T_{\psi}$ are invertible. Also, from Theroem \ref{theorem 1.3} \\
       $$\frac{\varphi}{|\varphi|}=e^{(\xi_{1}+\overline{\eta_{1}}+c_{1})}, \frac{\psi}{|\psi|} =e^{(\xi_{2}+\overline{\eta_{2}}+c_{2})}$$
       where $c_{1}, c_{2} \in \mathbb{R}$ and $\xi_{1},\xi_{2},\eta_{1},\eta_{2}\in L^{\infty}(\mathbb{D})$.\\
 As $w\neq 0$ and  $wT_{\varphi}+T_{\psi}$ is invertible   \begin{align*}
\implies \frac{(\varphi+w\psi)}{|\varphi+w\psi|} &=\frac{\varphi}{|\varphi+w\psi|}+\frac{w\psi}{|\varphi+w\psi|} \leq \frac{\varphi}{|\varphi|} +\frac{w\psi}{|w| |\psi|} =e^{\xi_{1}+\overline{\eta_{1}}+c_{1}}+\frac{w}{|w|}e^{\xi_{2}+\overline{\eta_{2}}+c_{2}}.
    \end{align*}
     \end{proof}
     \begin{theorem}
         Let \(\varphi(z) = \frac{z}{uz + v}\) with \(|v| \geq 1 + |u|\) and \(\psi(z) = \frac{z}{sz + t}\) with \(|t| \geq 1 + |s|\) for \(z \in \mathbb{D}\). If the operators $T_{\varphi}$ and $T_{\psi}$ are hyponormal and $\langle [T^*_{\varphi},T_{\psi}]K_{\alpha},K_{\alpha}\rangle =0 \text{ for all } \alpha \in \mathbb{D}$ then  $wT_{\varphi}+T_{\psi}$ is hyponormal where $w \in \mathbb{C}$ . Also,
         \begin{align*}
    \bigg|\frac{\overline{\alpha}}{(t\overline{\alpha} + s)} \bigg( \frac{1}{2}+ \frac{s}{2\overline{\alpha}t}+ \ln \bigg( \frac{\overline{\alpha u}}{\overline{v}}+1 \bigg) \frac{\overline{v}^2}{\overline{\alpha} ^2\overline{ u}^2}&+\frac{\overline{v}^2t} {\overline{\alpha}\overline{u}^2 s} \ln \bigg( 1- \frac{\overline{u}s}{\overline{v}t}\bigg)\bigg) \bigg| \leq \bigg|\frac{{v}}{\alpha u}\bigg|.
\end{align*}
     \end{theorem}
     \begin{proof}
         Let $e_{k}(z)= z^k$ for all $z \in \mathbb{D}.$ Since $\langle [T^*_{\varphi},T_{\psi}]K_{\alpha}, K_{\alpha}\rangle=0 \ \text{ for all } \alpha \in \mathbb{D} .$ Consider
         \begin{align*}
             0&= \langle[T^*_{\varphi},T_{\psi}]\sum_{k=0}^{\infty} \overline{\alpha}^{k} e_{k}, \sum_{j=0}^{\infty} \overline{\alpha}^j e_{j} \rangle  \\
             &= \sum_{k=0}^{\infty} \sum_{j=0}^{\infty} \overline{\alpha}^k \alpha^{j} \langle [T^*_{\varphi},T_{\psi}]e_{k},e_{j}\rangle.
         \end{align*}
         Set $\alpha=re^{i\theta}$ where $0<r<1 $ and $ \theta \in [0,2\pi] $ then for all $\alpha \in \mathbb{D}$.  
         \begin{align*}
             0&= \langle [T^*_{\varphi}, T_{\psi}] K_{\alpha}, K_{\alpha}\rangle e^{-i n\theta}\\
             & = \sum_{k=0}^\infty \sum_{j=0}^{\infty} r^{j+k} e^{i(j-k-n)\theta} \langle[T^*_{\varphi}, T_{\psi}]e_{k},e_{j}\rangle .
         \end{align*}
         Thus,
         \begin{align*}
               0&= \frac{1}{2\pi}\int_{0}^{2\pi} \langle [T^*_{\varphi},T_{\psi}]K_{\alpha},K_{\alpha}\rangle e^{-in\theta}d\theta\\
&=\sum_{k=0}^{\infty} r^{2k+n}\langle[T^*_{\varphi},T_{\psi}]e_{k},e_{k+n}\rangle.
         \end{align*}
 Therefore,
         $ \langle[T^*_{\varphi},T_{\psi}]e_{k},e_{k+n}\rangle=0 $
and so $\langle [T^*_{\varphi},T_{\psi}]e_{k},e_{m}\rangle=0 \text{ for all } \  k,m \in \mathbb{Z}$ where $0\leq k\leq m$. Also $\langle[T^*_{\varphi},T_{\psi}]^*K_{\alpha},K_{\alpha}\rangle =\overline{\langle[T^*_{\varphi},T_{\psi}]K_{\alpha},K_{\alpha}\rangle}=0$ for all $\alpha \in \mathbb{D}.$ Therefore, \\$ \langle[T^*_{\varphi},T_{\psi}]e_{k},e_{m}\rangle=0 $ for all $k, m \in \mathbb{Z}$. Thus, $\langle e_{k}, [T^*_{\varphi},T_{\psi}]e_{m}\rangle =0 $ \text{ for all }$ \  k,m \in \mathbb{Z}$ where $0\leq k\leq m$ and
$[T^*_{\varphi},T_{\psi}]=0.$ From equation \ref{(2.3)} we have, the operartor $wT_{\varphi}+T_{\psi} $ is hyponormal where $w \in \mathbb{C}.$\\
Consider for all $\alpha \in \mathbb{D}$
\begin{align}\label{2.3}
    T^*_{\varphi}T_{\psi}K_{\alpha} &\nonumber=T^*_{\varphi}P\bigg(\frac{z}{(sz+t)(1-\overline{\alpha} z)}\bigg)\\
    &\nonumber= \frac{1}{t}{T^*_{\varphi} {P\bigg(\sum_{k=1}^{\infty} \sum_{j=1}^k\bigg(\frac{-s}{t}\bigg)^{j-1} \overline{\alpha}^{k-j}\bigg)z^k}}\\
    & \nonumber=\frac{1}{t} P\bigg(\frac{\overline{z}}{\overline{uz}+\overline{v}}\bigg) \bigg(\sum_{k=1}^{\infty}\sum_{j=1}^k\bigg(\frac{-s}{t}\bigg)^{j-1} \overline{\alpha}^{k-j}\bigg)z^k\\
    &\nonumber=\frac{1}{t}P\bigg(\frac{1}{\overline{v}}\bigg(\sum_{n=1}^{\infty}\bigg(\frac{-\overline{u}}{\overline{v}}\bigg)^{n-1} \overline{z}^n\bigg) \bigg( \sum_{k=1}^{\infty}\bigg(\sum_{j=1}^{k}{\bigg(\frac{-s}{t}}\bigg)^{j-1} \overline{\alpha}^{k-j} z^k\bigg)\bigg)
    \end{align}
    \begin{align}
    & \hspace{1.5cm}\nonumber=\frac{1}{t{\overline{v}} }\sum_{n=1}^{\infty}\sum_{k=1}^{\infty}\sum_{j=1}^{k}\bigg(\frac{-\overline{u}}{\overline{v}}\bigg)^{n-1}\bigg(\frac{-s}{t}\bigg)^{j-1}\big(\overline{\alpha}\big)^{k-j}P(\overline{z}^n z^{k})\\
    &\hspace{1.5cm}=\frac{1}{2t\overline{v}}\sum_{n=1}^{\infty} \sum_{k=n}^{\infty} \sum_{j=1}^{k} \bigg(\frac{-s}{t}\bigg)^{j-1}  \big(\overline{\alpha}\big)^{k-j} \bigg(\frac{-\overline{u}}{\overline{v}}\bigg)^{n-1} \frac{z^{k-n}}{(k+1)}.
    \end{align}
Coefficient of $z$  is given by  $ \frac{1}{2\overline{v}t}\sum_{n=1}^{\infty}\sum_{j=1}^{n+1}  \bigg(\frac{-s}{t}\bigg)^{j-1} \frac{\overline{\alpha} ^{n+1-j}}{n+2} \bigg(\frac{-\overline{u}}{\overline{v}}\bigg)^{n-1}.$\\
Similarly 
\begin{align}\label{2.4}
    T_{\psi}T_{\varphi}^* K_{\alpha} &\nonumber= T_{\psi} P(\overline{\varphi(z)}K_{\alpha}(z))\\
    &=\frac{1}{2t\overline{v} }\sum_{p=1}^{\infty} \sum_{k=1}^{\infty} \sum_{n=k}^{\infty} \bigg(\frac{-s}{t}\bigg)^{p-1} \bigg(\frac{-\overline{u}}{\overline{v}}\bigg)^{k-1} \big(\overline{\alpha}\big)^{n} \frac{z^{n+p-k}}{(n+1)}. 
\end{align}
Coefficient of $z$  is given by $\frac{1}{2t\overline{v}}\sum_{n=1}^{\infty}\frac{\overline{\alpha}^{n}}{n+1} \bigg(\frac{-\overline{u}}{\overline{v}}\bigg)^{n-1}.$\\
Comparing coefficients of $z$ of equations \ref{2.3} and \ref{2.4}.
\begin{align}
\frac{1}{2t\overline{v}}\sum_{n=1}^{\infty}\frac{\overline{\alpha}^{n}}{n+1} \bigg(\frac{-\overline{u}}{\overline{v}}\bigg)^{n-1}= \frac{1}{2\overline{v}t}\sum_{n=1}^{\infty}\sum_{j=1}^{n+1}  \bigg(\frac{-s}{t}\bigg)^{j-1} \frac{\overline{\alpha} ^{n+1-j}}{n+2} \bigg(\frac{-\overline{u}}{\overline{v}}\bigg)^{n-1}.
\end{align}
Consider
\begin{align}
    \nonumber \frac{1}{2t\overline{v}}\sum_{n=1}^{\infty}\sum_{j=1}^{n+1}\bigg(\frac{-s}{t}\bigg)^{j-1} \frac{\big(\overline{\alpha}\big)^{n+1-j}}{n+2} &\bigg(\frac{-\overline{u}}{\overline{v}}\bigg)^{n-1}\\
    &\hspace{-2cm}\nonumber= \frac{ \overline{\alpha}}{2s\overline{u}}\sum_{n=1}^{\infty} \frac{(\overline{\alpha})^n}{n+2}\bigg(\frac{-\overline{u}}{\overline{v}}\bigg)^{n}\sum_{j=1}^{n+1} \bigg(\frac{-s}{t}\bigg)^{j}\frac{1}{(\overline{\alpha})^j} \\
    &\hspace{-2cm}\nonumber=-\bigg(\frac{ \overline{\alpha}}{2t\overline{u}}\bigg)\sum_{n=1}^{\infty} \bigg(\frac{-\overline{\alpha u}}{\overline{v}}\bigg)^n \frac{1}{n+2} \frac{(1-\big(\frac{-s}{t\overline{\alpha}}\big)^{n+1})}{(1+\frac{s}{t\overline{\alpha}})} \bigg(\frac{-s}{t\overline{\alpha}}  \bigg)\\
    & \hspace{-2cm}=-\frac{ \overline{\alpha} }{2\overline{u}(t\overline{\alpha}+s)}\bigg(\sum_{n=1}^\infty \bigg(\frac{-\overline{\alpha u}}{\overline{v}}\bigg)^n\frac{1}{n+2}+ \frac{s}{\overline{\alpha}t}\sum_{n=1}^\infty \bigg( \frac{\overline{u}s}{\overline{v}t}\bigg)^n \frac{1}{n+2}\bigg).  \nonumber
\end{align}
Therefore,
\begin{align}
\nonumber    \frac{1}{2t\overline{v}}\sum_{n=1}^{\infty}\sum_{j=1}^{n+1}\bigg(\frac{-s}{t}\bigg)^{j-1} \frac{\big(\overline{\alpha}\big)^{n+1-j}}{n+2} &\bigg(\frac{-\overline{u}}{\overline{v}}\bigg)^{n-1}\\
    &\hspace{-2cm}=- \frac{  \overline{ \alpha}}{2(t\overline{\alpha}+s)\overline{u}} \bigg( \frac{1}{2} + \frac{s}{2\overline{\alpha}t }+ \ln \bigg( \frac{\overline{\alpha u }}{\overline{v}} +1 \bigg) \frac{\overline{v}^2}{\overline{\alpha}^2 \overline{ u }^2 } +\frac{\overline{v}^2 t}{\overline{\alpha} \overline{u}^2 s} \ln \bigg( 1- \frac{\overline{u }s}{\overline{v}t}\bigg)\bigg).\label{2.6}
\end{align}
Also,
\begin{align}
\frac{1}{2\overline{v}t}\sum_{k=1}^{\infty}\frac{(\overline{\alpha})^k}{k+1} \bigg(\frac{-\overline{u}}{\overline{v}}\bigg)^{k-1} = \frac{1}{2t \overline{u}}\bigg(1-\frac{\overline{v}}{\overline{u \alpha}} \ln \bigg( 1+\frac{\overline{\alpha u}}{\overline{v}}\bigg) \bigg) .\label{2.7}
\end{align}
Equating equations \ref{2.6} and \ref{2.7} and from inequalities $|v| \geq 1+ |u| $, $|t| \geq 1+ |s|$ and $|\ln(1+\frac{\overline{\alpha u}}{\overline{v}})| \leq |1+\frac{\overline{\alpha u}}{\overline{v}}|$.
\begin{align}
    \bigg|\frac{\overline{ \alpha}}{(t\overline{\alpha} + s)} \bigg( \frac{1}{2}+ \frac{s}{2\overline{\alpha}t}+ \ln \bigg( \frac{\overline{\alpha u}}{\overline{v}}+1 \bigg) \frac{\overline{v}^2}{\overline{\alpha}^2 \overline{ u}^2}&+\frac{\overline{v}^2t} {\overline{\alpha}\overline{u}^2 s} \ln \bigg( 1- \frac{\overline{u}s}{\overline{v}t}\bigg)\bigg) \bigg|\leq  \bigg|\frac{{v}}{\alpha u}\bigg|.
\end{align}


     \end{proof}
     \section{Normal Operator}
A bounded linear operator  $R$  on a Hilbert space $\mathcal{H}$  is termed as normal if $[R^*, R] = 0$. In this section we have studied the normality of the sum of the two Toeplitz operators of the type $wT_{\varphi}+T_{\psi} $ where $w$ is any non-zero complex number and the functions $\varphi, \psi \in L^{\infty}(\mathbb{D})$.
\begin{theorem}
          Let $\varphi$, $ \psi$ $\in L^{\infty}(\mathbb{D})$ be such that $w T_{\varphi}+ T_{\psi}$ is normal for any non-zero $w \in \mathbb{C}$ if and only if $T_{\varphi}$ and $T_{\psi}$ are normal and $ Re \{  \overline{w} \langle[T_{\varphi}^*, T_{\psi}]h, h \rangle \} =0$ for all $h \in H^{2}(\mathbb{D} ).$ 
     \end{theorem}
     \begin{proof}
         Let $wT_{\varphi}+T_{\psi}$ be normal operator then 
         \begin{align}
        &\nonumber\implies [(wT_{\varphi} + T_{\psi})^*, wT_{\varphi} + T_{\psi}] = 0\\
         &\nonumber \iff |w|^2(T^*_{\varphi}T_{\varphi}-T_{\varphi}T^*_{\varphi})+\overline{w} (T^*_{\varphi}T_{\psi}-T_{\psi}T^*_{\varphi}) + w(T^*_{\psi}T_{\varphi}-T_{\varphi}T^*_{\psi})+(T^*_{\psi}T_{\psi}-T_{\psi}T^*_{\psi}) = 0\\
        & \label{3.1} \iff |w|^2 [T^*_{\varphi},T_{\varphi}]+\overline{w}[T^*_{\varphi},T_{\psi}]+w[T^*_{\psi},T_{\varphi}]+[T^*_{\psi},T_{\psi}] = 0.
    \end{align}
        Let $w= re^{i \theta}$ where $r>0$ and $\theta \in [0, 2\pi] $
        \begin{align}
        & \iff r^2 [T^*_{\varphi},T_{\varphi}]+re^{-i \theta} [T^*_{\varphi}, T_{\psi}] + re^{i \theta} [T^*_{\psi}, T_{\varphi}] +[T^*_{\psi}, T_{\psi}] = 0 \nonumber\\
        & \iff [T^*_{\varphi},T_{\varphi}] + \frac{1}{r^2}[T^*_{\psi},T_{\psi}]+ \frac{e^{-i \theta}}{r} [T^*_{\varphi},T_{\psi}]+\frac{e^{i \theta}}{r}[T^*_{\psi},T_{\varphi}] = 0
    \end{align} 
    Letting $r \rightarrow  \infty $ we have
    $[T^*_{\varphi}, T_{\varphi}] = 0 $ $\implies T_{\varphi}$ is normal.
    Since $[T^*_{\psi} , T_{\varphi}]=[T^*_{\varphi} ,T_{\psi}]^*.$
    \begin{align}
      & \nonumber |w|^2[T^*_{\varphi} ,T_{\varphi}]+[T^*_{\psi} ,T_{\psi}]+\overline{w}[T^*_{\varphi} ,T_{\psi}]+w[T^*_{\psi} ,T_{\varphi}] = 0.\\
      & \nonumber \implies |w|^2[T^*_{\varphi} ,T_{\varphi}]+[T^*_{\psi} ,T_{\psi}] +2 Re \{\overline{w}[T^*_{\varphi} ,T_{\psi}]\} = 0.\\
      & \implies |w|^2 \langle [T^*_{\varphi} ,T_{\varphi}]h,h\rangle +\langle [T^*_{\psi} ,T_{\psi}] h, h \rangle +  2 Re \{\overline{w} \langle [T^*_{\varphi} ,T_{\psi}]h,h\rangle \} = 0  \ \ \text{ for all } h \in H^2(\mathbb{D} ).\label{3.9}
    \end{align}
    Since  $[T^*_{\varphi} ,T_{\varphi}]$ and $[T^*_{\psi} ,T_{\psi}]$ are self-adjoint both $\langle [T^*_{\varphi} ,T_{\varphi}] h, ,h \rangle $ and $\langle [T^*_{\psi} ,T_{\psi}] h, h \rangle $ are real for all $h \in H^2(\mathbb{D} )$. Therefore,\\
$$|\langle [T^*_{\varphi},T_{\psi}]h,h\rangle |^2 \leq \langle [T^*_{\varphi},T_{\varphi}]h,h\rangle \langle [T^*_{\psi},T_{\psi}]h,h\rangle \text{ for all } h \in H^2(\mathbb{D} )$$
    Since the operator $T_{\varphi}$ is normal. Therefore, $ Re \{\overline{w}\langle [T^*_{\varphi} ,T_{\psi}]h , h \rangle \}=0 $ for all $h \in {H}^2(\mathbb{D} ).$ Hence, from equation  \ref{3.9} $[T^*_{\psi},T_{\psi}]=0$.\\
    Conversely, assume $T_{\varphi}$ and $T_{\psi}$ are normal then $ [T^*_{\varphi},T_{\varphi}],  [T^*_{\psi},T_{\psi}] = 0.$\\
    Consider
    \begin{align}
       \nonumber& [(w T_{\varphi}+T_{\psi})^* , w T_{\varphi}+T_{\psi}]\\
       \nonumber&\hspace{4cm}= (w T_{\varphi}+T_{\psi})^*(w T_{\varphi}+T_{\psi})- (w T_{\varphi}+T_{\psi})(w T_{\varphi}+T_{\psi})^*\\
       &\hspace{4cm}=|w|^2 [T^*_{\varphi},T_{\varphi}]+ [T^*_{\psi},T_{\psi}] + 2 Re\{ \overline{w} [T^*_{\varphi},T_{\psi}]\}.\label{(3.3)}
    \end{align}
If $Re \{\overline{w} [T^*_{\varphi},T_{\psi}]\} = 0 $ 
then $ [(w T_{\varphi}+T_{\psi})^*, w T_{\varphi} +T_{\psi} ] = 0. $
Hence, the operator $w T_{\varphi}+T_{\psi} $ is normal.
     \end{proof}
 If we drop the normality of any one of the operators in the above theorem then the result does not holds. The following example verifies the same:
 \begin{example}
     Let $\varphi(z) = z+ \overline{z} $ ,  $\psi(z) = z + z^2$ be in $ L^\infty (\mathbb{D})$ . Consider
\begin{align*}
         & T_{\psi}(z)=z^2+z^3,  T^*_{\psi}(z)= \frac{1}{4}, T_{\varphi}(z)= z^{2}+\frac{1}{4} . 
        \end{align*}
        $\text{ As  $T_{\varphi} $ is self-adjoint.} \text{ Therefore, } [T^*_{\varphi}, T_{\varphi}] = 0.$ From calculations, we get,
         \begin{align*}
&\langle T_{\psi}(z), T_{\psi}(z) \rangle = \frac{7}{24}, \langle T^*_{\psi}(z), T^*_{\psi}(z) \rangle = \frac{1}{32}  \text { and } \langle T_{\varphi}(z), T_{\varphi}(z) \rangle = \frac{19}{96}, \\
& \langle T_{\psi}(z), T_{\varphi}(z) \rangle - \langle  T^*_{\varphi}(z), T_{\psi}(z) \rangle  = \frac{13}{96}, \langle T_{\varphi}(z), T_{\varphi}(z) \rangle - \langle  T^*_{\varphi}(z), T_{\varphi}(z) \rangle  = \frac{25}{96}.
   \end{align*}
For $w=-1$. we have, $\langle (\overline{w} T^*_{\varphi}+T^*_{\psi} )(z), (wT_{\varphi}+T_{\psi}) (z) \rangle = \frac{17}{32}.$ Therefore, the operator $wT_{\varphi}+T_{\psi}$ is non-normal for $w=-1$.  
\end{example}
\begin{corollary}
 Let $\varphi$ and $ \psi$ be analytic maps  from $\mathbb{D}$ into itself and  $T_{\varphi}$, $T_{\psi}$ are normal then $wT_{\varphi}+T_{\psi}$ is also normal where $0\neq w \in \mathbb{C}$.
 \end{corollary}
 \begin{proof}
 As $\varphi$ and $\psi$ are analytic, $T_{\varphi} T_{\psi} = T_{\varphi \circ \psi} = T_{\psi \circ \varphi} = T_{\psi} T_{\varphi}.$ 
 By Fuglede-Putnam theorem and equation \ref{2.1}, $[T^* _{\varphi}, T_{\psi}]=0.$
 Hence,
 $ [(wT_{\varphi} +T_{\psi} )^*, wT_{\varphi} +T_{\psi}] =0.$
 \end{proof}


\subsection*{Data Availability }
Data availability not applicable. 

\subsection*{Financial disclosure}
The second author is thankful to the Council of Scientific and Industrial Research (CSIR) [Grant number: 09/0045(13796)/2022-I] for their support in the form of a grant-in-aid for this research.

\subsection*{Conflict of interest}

The authors declare no potential conflict of interests.\\
\nocite{*}


\bibliography{ref.bib}{}
\bibliographystyle{plain}

\noindent \textbf{Anuradha Gupta}\\
 Department of Mathematics, Delhi College of Arts and Commerce,\\
  University of Delhi, Netaji Nagar, \\
  New Delhi-110023, India.\\
  email: dishna2@yahoo.in\\
  \vspace{0.2cm}

\noindent \textbf{Kajal Negi}\\
  Department of Mathematics,\\
  University of Delhi, \\
  New Delhi-110007, India.\\
  email: kajalnegi1109@gmail.com
\end{document}